\documentclass{amsart}

\usepackage{amsmath,amsthm,amsfonts,amssymb,times,latexsym,comment,marginnote,slashed,dsfont}
\usepackage{pdfpages, relsize, setspace}
\usepackage{xypic}
\usepackage[OT2,T1]{fontenc} 
\usepackage{listings}
\usepackage{tabu}
\usepackage{stmaryrd}
\usepackage{tikz-cd}
\usepackage{bbm}

\newcommand{\mbz}{\mathbb{Z}}

\newcommand{\mbq}{\mathbb{Q}}
\newcommand{\mba}{\mathbb{A}}

\newcommand{\mbp}{\mathbb{P}}

\renewcommand{\:}{\colon}

\newcommand{\ra}{\rightarrow}
\newcommand{\iso}{\cong}

\newcommand{\wtilde}{\widetilde}

\DeclareMathOperator{\Pic}{Pic}

\DeclareMathOperator{\Aut}{Aut}

\DeclareMathOperator{\Hom}{Hom}

\DeclareMathOperator{\Gal}{Gal}

\DeclareMathOperator{\im}{Im}
\DeclareMathOperator{\Stab}{Stab}

\DeclareMathOperator{\Jac}{Jac}

\renewcommand{\Im}{\im}

\DeclareMathOperator{\Spec}{Spec}

\DeclareMathOperator{\id}{id}

\newcommand*{\longhookrightarrow}{\ensuremath{\lhook\joinrel\relbar\joinrel\rightarrow}}
\newcommand*{\inj}{\longhookrightarrow}

\newcommand{\xym}[1]{\xymatrix{#1}}

\newcommand{\mapdef}[5]{
	\begin{tabu}{cccc}
		#1\: & #2 & \ra & #3 \\
		& #4 & \mapsto & #5
	\end{tabu}
}

\newcommand{\floor}[1]{\left \lfloor #1 \right \rfloor}

\newcommand{\gen}[1]{\langle #1 \rangle}

\newcommand{\mcL}{\mathcal{L}}

\newcommand{\sep}{\mathrm{sep}}

\mathchardef\mhyphen="2D

\DeclareMathOperator{\Sym}{Sym}

\newcommand{\timestwo}{{\times 2}}

\newcommand{\arhook}{\ar@{^{(}->}}

\xyoption{rotate}

\DeclareMathOperator{\GL}{GL}

\newtheorem{theorem}{Theorem}[section]
\newtheorem{lemma}[theorem]{Lemma}
\newtheorem{proposition}[theorem]{Proposition}
\newtheorem*{theorem*}{Theorem}
\newtheorem{corollary}[theorem]{Corollary}

\makeatletter
\newtheorem*{rep@theorem}{\rep@title}
\newcommand{\newreptheorem}[2]{%
	\newenvironment{rep#1}[1]{%
		\def\rep@title{#2 \ref{##1}}%
		\begin{rep@theorem}}%
		{\end{rep@theorem}}}
\makeatother

\newreptheorem{theorem}{Theorem}
\newreptheorem{conjecture}{Conjecture}

\theoremstyle{definition}
\newtheorem{definition}[theorem]{Definition}
\newtheorem{definition-theorem}[theorem]{Definition-Theorem}
\newtheorem{definition-corollary}[theorem]{Definition-Corollary}

\newtheorem{example}[theorem]{Example}

\usepackage{thmtools} 
\usepackage{thm-restate}
\usepackage{enumerate}

\usepackage[alphabetic]{amsrefs}
\usepackage{mathtools}
\numberwithin{theorem}{subsection}

\newtheorem{construction}[theorem]{Construction}

\usepackage{hyperref}
\hypersetup{
    colorlinks=true,
    linkcolor=blue,
    citecolor=blue,
    filecolor=magenta,      
    urlcolor=cyan,
    pdftitle={Sixteen points in $\mbp^4$ and the inverse Galois problem for del Pezzo surfaces of degree one},
    pdfpagemode=FullScreen,
    }
    
\usepackage{url}            % simple URL typesetting
\usepackage{booktabs}       % professional-quality tables

\usepackage{tikz}
\usetikzlibrary{matrix,arrows}

\usepackage[top=3cm,bottom=2cm,left=3cm,right=3cm,marginparwidth=1.75cm]{geometry}

\usepackage[colorinlistoftodos]{todonotes}

\newcommand{\defi}[1]{\emph{#1}}

\renewcommand{\H}{\operatorname{H}}

\newcommand{\dual}{\widehat}

\DeclareMathOperator{\Proj}{Proj}
\DeclareMathOperator{\gencover}{gen}

\newcommand{\WD}{W_{D_8}}
\DeclareMathOperator{\Br}{Br}

\title{Sixteen points in $\mbp^4$ and the inverse Galois problem for del Pezzo surfaces of degree one}

\author{Avinash Kulkarni}
\address{Department of Mathematics, Dartmouth College}
\email{avinash.a.kulkarni@dartmouth.edu}


\begin{document}

\maketitle

\begin{abstract}
	A del Pezzo surface of degree one defined over the rationals has 240 exceptional curves. These curves are permuted by the action of the absolute Galois group. We show how a solution to the classical inverse Galois problem for a subgroup of the Weyl group of type $D_8$ gives rise to a solution of the inverse Galois problem for the action of this subgroup on the 240 exceptional curves. A del Pezzo surface of degree one with such a Galois action contains a Galois invariant sublattice of type $D_8$ within its Picard lattice; this can be characterized in terms of a certain set of sixteen points in $\mathbb{P}^4$.
\end{abstract}

%%% SECTION
\section{Introduction}

	If $\wtilde X_2$ is a del Pezzo surface of degree one, then it is the double cover of a quadric cone $X_2$ branched over a nonsingular genus four curve $X$ and the vertex of the cone. The quadric cone $X_2$ containing $X$ identifies a vanishing even theta characteristic $\theta_0$ of $X$. 
	Let $\kappa$ denote the canonical class of $\wtilde X_2$. The intersection pairing on $\wtilde X_2$ gives $\Pic(\wtilde X_2)$ the structure of a lattice, which admits a decomposition of the form
	\[
		\Pic(\wtilde X_2) = \gen{\kappa} \oplus \Pic(\wtilde X_2)^\perp
	\]
	where $\Pic(\wtilde X_2)^\perp$ denotes the sublattice of divisor classes orthogonal to the canonical class under the intersection pairing. The sublattice $\Pic(\wtilde X_2)^\perp$ is a root lattice of type $E_8$, and such a lattice has $135$ sublattices $\Lambda_{D_8} \subset \Pic(\wtilde X_2)^\perp$ of type $D_8$. In this article, we consider those del Pezzo surfaces of degree one whose Picard lattice has a Galois invariant sublattice of type $D_8$. The Galois action on $\Pic(\wtilde X_2)$ permutes the $240$ exceptional curves of $\wtilde X_2$ and acts on $\Pic(\wtilde X_2)^\perp$ though the Weyl group of type $E_8$, denoted $W_{E_8}$. As it turns out, the stabilizer in the Weyl group $W_{E_8}$ of a $D_8$-sublattice is of index $135$, and in fact it is isomorphic to the Weyl group $W_{D_8}$. Further details about del Pezzo surfaces can be found in \cite[Chapter~8]{Dol2012}. We show how to transmute the solution to the classical inverse Galois problem for subgroups of $\WD$ into a solution to the inverse Galois problem for del Pezzo surfaces of degree one with a Galois invariant sublattice $\Lambda_{D_8} \subset \Pic(\wtilde X_2)$ of type $D_8$.

	\begin{theorem} \label{thm: main theorem}
		Let $\rho\: \Gal(\mbq^\sep/\mbq) \rightarrow \WD$ be a continuous homomorphism. If $G := \Im(\rho)$ is the Galois group of some irreducible polynomial over $\mbq$, then there exists a nonsingular del Pezzo surface of degree one $\wtilde X_2$ such that each $\sigma \in \Gal(\mbq^\sep/\mbq)$ permutes the $240$ exceptional curves of $\wtilde X_2$ as described by $\rho(\sigma) \subset S_{240}$.
	\end{theorem}

	In particular, the subfield of $\mbq^\sep$ fixed by $\rho^{-1}(\id)$ is the splitting field of the $240$ exceptional curves. The group $\WD$ admits a transitive permutation representation on a set of size $16$ (see Section~\ref{sec: background}). If $\rho'\: \WD \rightarrow S_{16}$ is the associated morphism and if $G \subseteq \WD$ is a subgroup for which a solution to the inverse Galois problem is known, then there exists a (possibly reducible) polynomial $f$ of degree $16$ such that $G \iso \Gal(f)$ and each $\sigma \in G$ permutes the roots of $f$ as described by $\rho'(\sigma)$. A solution to the inverse Galois problem is known for several subgroups of $\WD$, including $\WD$ itself.

	\begin{theorem} \label{thm: inverse Galois polynomials}
		Let $G$ be a group which acts transitively on a set of size at most $16$. Then there is an irreducible polynomial $f$ over $\mbq$ such that $\Gal(f) \iso G$.
	\end{theorem}
	
	\begin{proof}
		A list of examples can be obtained from the online database of Kl\"uners and Malle \cites{KlunersMalleDatabaseArticle, KlunersMalleDatabase2021}.
	\end{proof}

	Our techniques are related to those in the article by Elsenhans and Jahnel~\cite{ElsenhansJahnel2019}, where they study plane quartics in terms of a set of eight points in $\mbp^3$ called a Cayley octad.

	In \cite{Corn2007}, Corn showed that the Brauer group of a del Pezzo surface is one of a finite number of explicit possibilities. Specifically, for a del Pezzo surface $X$ over $\mbq$ there is a canonical isomorphism $\Br X / \Br \mbq \iso \H^1(\Gal(\mbq^\sep/\mbq), \Pic X^\sep)$ coming from the Hochschild-Serre spectral sequence. Corn's result is obtained by considering the action of an abstract group on the Picard lattice and using group cohomology to determine the list of possibilities. Corn's result raises an interesting question, \emph{which Brauer groups from Corn's list actually occur for a del Pezzo surface defined over $\mbq$?}
	Theorem~\ref{thm: main theorem} allows us to resolve this question for the groups in Corn's list \cite[Theorem~4.1]{Corn2007} of exponent $1$, $2$, or $4$; we list these in Table~\ref{tbl: Corn list}.

	\begin{table}
		\label{tbl: Corn list}
		\begin{center}
		\begin{tabular}{rl} 
			Degree & Groups of exponent $1,2,4$ \\ \hline \\[-1em]
			$5 \leq d \leq 9$ & $\{\id\}$ \\
			$d=4$ & any of the groups above, $\mbz/2\mbz, (\mbz/2\mbz)^2$ \\
			$d=3$ & any of the groups above \\
			$d=2$ & any of the groups above, $(\mbz/2\mbz)^s \ \ (3 \leq s \leq 6)$, $\mbz/4\mbz \oplus (\mbz/2\mbz)^t \ \ (0 \leq t \leq 2)$, $(\mbz/4\mbz)^2$ \\
			$d=1$ & any of the groups above, $(\mbz/2\mbz)^s \ \ (7 \leq s \leq 8)$, $\mbz/4\mbz \oplus (\mbz/2\mbz)^t \ \ (0 \leq t \leq 4)$, \\ & $(\mbz/4\mbz)^2 \oplus (\mbz/2\mbz)^t \ \ (1 \leq t \leq 2)$ \\
		\end{tabular}
		\bigskip
		\caption{The list of possible $\Br X/\Br \mbq$ for del Pezzo surfaces by degree.}	
		\end{center}
	\end{table}

	\begin{corollary}
		Any group from row $d$ in Table~\ref{tbl: Corn list} is of the form $\Br X/\Br \mbq$ for some del Pezzo surface $X/\mbq$ of degree~$d$.
	\end{corollary}
	
	\begin{proof}
		First we consider the case where $d=1$. If $G$ is a subgroup of $\WD$ for which a solution to the inverse Galois problem is known, then the action of $G \subseteq \WD \subset W_{E_8}$ on the exceptional curves of a del Pezzo surface of degree one can be realized as a Galois action over $\mbq$. We can explicitly compute the group cohomology for the action of $G$ on the root lattice $\Lambda_{E_8}$ of type $E_8$ and check for the desired Brauer group. A script that does this for all of the subgroups of $\WD$ is available at the link in Subsection~\ref{sec: sub: software}. (Of course, Corn carries out a similar calculation to prove \cite[Theorem~4.1]{Corn2007}.) For each of the groups $M$ in Table~\ref{tbl: Corn list}, our computation shows there is a subgroup $G \subset \WD$ of order at most $8$ such that $M \iso \H^1(G, \Lambda_{E_8})$. In particular, a solution to the inverse Galois problem for such $G$ is known.
		
		For del Pezzo surfaces of degree $2 \leq d \leq 4$, we may apply the result of \cite{ElsenhansJahnel2019}. Alternatively, we search for a $G$-action on $\Pic(\wtilde X_2)^\perp$ such that there is an isomorphism of $G$-lattices
		\[
			\Pic(\wtilde X_2)^\perp \iso \Lambda \oplus \gen{e_1} \oplus \cdots \oplus \gen{e_{d-1}}
		\]
		where $e_j^2 = -1$, the summands are orthogonal, and where $\H^1(G, \Lambda)$ is as desired. A del Pezzo surface of degree one with such a Picard lattice has a set of $d-1$ pairwise orthogonal exceptional curves defined over $\mbq$; these can be blown-down to obtain a del Pezzo surface of degree $d$ with Brauer group $\H^1(G, \Lambda)$.
	\end{proof}

	\subsection{Software} \label{sec: sub: software}
		The computations for this paper were performed using the \texttt{magma} computer algebra system \cite{Magma}. The scripts are available at
		\begin{center}
			\url{https://github.com/a-kulkarn/SixteenPointsScripts.git}
		\end{center}
		The total computation time to run all of the scripts is approximately $1$ hour on commercial hardware (processor: AMD Ryzen$^{\text{tm}}$ Threadripper$^{\text{tm}}$ 2970WX).
				
		% Doob: AMD Opteron$^{\text{tm}}$ 6380 processor
		% Toby: AMD Ryzen Threadripper 2970WX 24-Core Processor

	%%% SUBSECTION
	\subsection*{Acknowledgements}
		I would like to thank Eran Assaf for comments on the early drafts of this article.
		The author has been supported by the Simons Collaboration on Arithmetic Geometry, Number Theory, and Computation (Simons Foundation grant 550033).

%%% SECTION
\section{Background} \label{sec: background}

\begin{comment}
\begin{corollary}
	Let $L/\mbq$ be a number field whose Galois group is a subgroup of $\WD$ and let $K$ be the normal $S_8$-subextension. Then there exists a normal basis $\alpha_1, \ldots, \alpha_8$ of $K$ such that $\mbq(\alpha_1, \ldots, \alpha_8)$ generate the $S_8$ subextension $K$ and $L := K(\sqrt{\alpha_1}, \ldots, \sqrt{\alpha_8})$.
\end{corollary}

\begin{proof}
	Kummer theory.
\end{proof}
\end{comment}

Let $K/\mbq$ be an \'etale algebra, and let $K = \prod_{j} K_j$ be a decomposition into simple factors. If $K$ is simple, then the \defi{norm} of an element $\alpha \in K$ is the usual field norm $N_{K/\mbq}$. Otherwise, we define the norm of $\alpha$ by $N_{K/\mbq}(\alpha) := \prod_j N_{K_j/\mbq}(\alpha_j)$.

\subsection{\'Etale algebras with a $\WD$ action}

	The (abstract) group $\WD$ is isomorphic to an index $2$ subgroup of the wreath product $S_2 \wr S_8$. Specifically, it is the extension of $S_8$ by the subgroup of elements of $\mu_2^8$ whose entries multiply to $1$. We have the exact sequence
	\[
		\xym{
			0 \ar[r] & \mu_2^7 \ar[r] & \WD \ar[r] & S_8 \ar[r] & 1 
		}.
	\]
	Essentially by definition, the wreath product $S_2 \wr S_8$ admits an action on a set of $16$ elements $\Omega$ (the eight elements on which $S_8$ acts endowed with signs). 
	
	\begin{construction} \label{cons: etale algebra of degree 16}
	Given $\rho\: \Gal(\mbq^\sep/\mbq) \rightarrow \WD$ a continuous homomorphism with image $G$, we describe how to construct an \'etale algebra $L$ of degree $16$ such that $\Hom(L, \mbq^\sep)$ is isomorphic to $\Omega$ as a $G$-set. We may choose a set of orbit representatives $\beta_1, \ldots, \beta_r$ such that $\Omega = G\cdot \beta_1 \cup \cdots \cup G \cdot \beta_r$. Functoriality between \'etale algebras and continuous Galois actions on a finite set allows us to identify an \'etale algebra $\widehat L$ from $\rho$ \cite[Chapter 8]{milneFT}. Functoriality also allows us to define an \'etale subalgebra $L$ of degree $16$ as the subalgebra element-wise fixed by $\Stab(G; \beta_1, \ldots, \beta_r)$.
	\end{construction}
	
	The $16$ elements of $\Omega$ are identified with the $16$ homomorphisms $L \rightarrow \mbq^\sep$ via the $\Hom$ functor. The natural quotient of $G$-sets $\Omega \rightarrow (\Omega/{\pm})$ induces an inclusion of \'etale algebras $K \inj L$, where the degree of $K$ is equal to $8$. Up to isomorphism we may write $K = \prod_j K_j$ and
	\begin{equation} \label{eq: L presentation}
			L \iso \prod_{j} K_j[x]/(x^2 - \alpha_j) \tag{$\dagger$}
	\end{equation}
	for some finite field extensions $K_j/\mbq$ and some $\alpha_j \in K_j$ with the property that $\prod_j N_{K_j/\mbq}(\alpha_j) \in \mbq^\timestwo$. We call $K$ the \defi{distinguished subalgebra}.

%%% SUBSECTION		
\subsection{Zariski density of generators}

\newcommand{\Kscheme}{\mba^{\!K}}
\newcommand{\KschemeA}{\mba^{\!A}}
\newcommand{\Poleight}{\mathrm{Pol}^8}
\newcommand{\sq}{\mathrm{sq}}

We denote by $\mathrm{Pol}^8$ the scheme whose set of $R$-valued points is the set of monic polynomials of degree $8$ with coefficients in $R$, where $R$ is a commutative ring containing $\mbq$. Of course, $\mba^8_\mbq \iso \Poleight$. 
If $K$ is an \'etale $\mbq$-algebra of dimension $8$, we may endow $K$ with the structure of an affine scheme, specifically the affine space $\mba^8_\mbq$. We denote this scheme by $\Kscheme$. The functor of points is given by $\Kscheme\!(R) = K \otimes_\mbq R$. 
%
%To be precise about the functor of points, if $\gamma_1, \ldots, \gamma_8$ is a fixed basis for $K/\mbq$, then we define the $R$-valued points by
%\[
%	\Kscheme\!(R) = \{r_1 \gamma_1 + \ldots + r_8 \gamma_8 : r_1, \ldots, r_8 \in R\}.
%\]
%for any affine scheme $\Spec R$ over $\mbq$.
%
%Note that the set of points is independent of the choice of basis and by definition, $\Kscheme\!(\mbq) = K$ as a set. 
The scheme $\Kscheme$ is simply the Weil restriction $\mathrm{Res}_{K/\mbq} \mba^1_K$. If $R$ is a commutative ring containing $\mbq$, then we denote by $\Kscheme_R$ the scheme $\Kscheme \times_{\Spec \mbq} \Spec R$.
If $\beta \in \Kscheme\!(R)$, we denote by $[\beta]$ the endomorphism of $R \otimes_\mbq K$ defined by multiplication-by-$\beta$. There is a canonical morphism of schemes
\[
	\mapdef{\chi^K}{\Kscheme}{\Poleight}{\beta}{\mathrm{coeffs}(\det(xI - [\beta]))}.
\]
%Note that $\chi^K$ is independent of the choice of basis $\gamma_1, \ldots, \gamma_8$. 
As an example, if $K = \prod_{j=1}^8 \mbq$ is the split \'etale $\mbq$-algebra of dimension $8$, then $\chi^K$ is simply the map defined by the elementary symmetric functions.

%We denote by $\Kscheme$ the functor parametrizing the characteristic polynomials of elements of $K$. Note that for a dense subfunctor, the characteristic polynomials agree with the minimal polynomials.

\begin{lemma} \label{lem: minimal polynomials are dense}
	The set $\chi^K(\Kscheme\!(\mbq))$ is Zariski dense in $\Poleight$.
\end{lemma}

\begin{proof}
	Let ${\widehat K}$ be a splitting field for $K$. Since $\Kscheme$ is just affine space, $\Kscheme\!(\mbq)$ is dense in $\Kscheme_{\widehat K}$. On the other hand, if $A := {\mbq}^8$ is the split \'etale ${\widehat K}$ algebra, diagonalization gives an isomorphism $\psi\: \Kscheme_{\widehat K} \rightarrow \KschemeA_{\widehat K}$. Diagonalization does not change the characteristic polynomial, so the diagram
	\[
		\xym{
			\Kscheme_{\widehat K} \ar[d]_-{\chi^K_{\widehat K}} \ar[r]^\psi & \KschemeA_{\widehat K} \ar[dl]^-{\chi^A_{\widehat K}} \\
			\Poleight_{\widehat K}
		}
	\]
	of morphisms over ${\widehat K}$ commutes. But $\chi^A$ is just the morphism defined by elementary symmetric functions, so it is surjective as a morphism of schemes, and thus so too is $\chi^K$.
\end{proof}

\begin{corollary} \label{cor: sqrt primitive basis}
	Given an \'etale algebra $L$ arising from Construction~\ref{cons: etale algebra of degree 16}, with distinguished subalgebra $K$, the set
	\[
		\left\{(\alpha_1, \ldots, \alpha_r) \in \prod_j K_j : L \iso \prod_j {K_j[x]}/{\gen{x^2-\alpha_j}}, \ K_j \iso \mbq(\alpha_j)\right\}
	\]
	has a dense image under $\chi^K$ in $\Poleight$. 
\end{corollary}

\begin{proof}
	Choose a presentation for $L$ as in Equation~\eqref{eq: L presentation}. Since the set $\alpha \cdot K^\timestwo$ is Zariski dense in $\Kscheme$ for any unit $\alpha \in K^\times$, and contains a dense subset of primitive elements, the result follows.	
\end{proof}

%%% Subsection
\subsection{Tensors}

We denote by $\Sym_2 \mbq^{m+1}$ the space of $(m+1) \times (m+1)$ symmetric matrices over $\mbq$ and by $\Sym^2 \mbq^{n+1}$ the space of quadratic forms over $\mbq$ in $n+1$ variables.
Let $\mathcal{A} \in \mbq^{n+1} \otimes \Sym_2 \mbq^{m+1}$ be a tensor, symmetric in the last two entries. We view $\mathcal{A}$ as an $(n+1) \times (m+1) \times (m+1)$ array of elements of $\mbq$. We may think of such an array as being an ordered collection $A_0, \ldots, A_n$ of symmetric $(m+1) \times (m+1)$ matrices, the \emph{slices} of $\mathcal{A}$. We denote the contraction of $\mathcal{A}$ along a vector $v \in \mbq^{n+1}$ by $\mathcal{A}(v, \cdot, \cdot)$. More generally, we will contract along an element of $R^n$, where $R$ is a $\mbq$-algebra. Similarly, we will denote the contraction by an element $y \in R^{m+1}$ by $\mathcal{A}(\cdot, y, \cdot)$ or $\mathcal{A}(\cdot, \cdot, y)$, depending along which axis we contract. 

Denote $\mbp^{n} := \Proj \mbq[x_0,\ldots,x_n]$, $\mbp^m := \Proj \mbq[y_0, \ldots, y_m]$,  and denote the dual projective space of $\mbp^n$ by $\widehat \mbp^n$. We will also denote $\mathbf{x} := (x_0, \ldots, x_n)$ and $\mathbf{y} := (y_0, \ldots, y_m)$. 
%The contractions $\mathcal{A}(\mathbf{x}, \cdot, \cdot), \mathcal{A}(\cdot, \mathbf{y}, \cdot)$ are matrices of linear forms. 
If $x \in \mbq^{n+1}$, the contraction $\mathcal{A}(x, \mathbf{y}, \mathbf{y})$ is the quadratic form
\[
	\mathbf{y}^T\left(x_0A_0 + \ldots + x_nA_n \right)\mathbf{y}.
\]
If $x \in \mbp^n(\mbq)$ is a point, then the contraction
	$
		\mathcal{A}(x, \cdot, \cdot)
	$
is a symmetric $(m+1) \times (m+1)$ matrix with entries in $\mbq$, well-defined up to scaling. 
%Such a symmetric matrix is associated to the defining equation for a quadric in $\mbp^m$. 
Write $(q_0, \ldots, q_n)$ for the quadrics defined by the $n+1$ slices of $\mathcal{A}$. The $q_i$ define the rational map
	\[
		\begin{tabu}{cccc}
			\psi\: & \mbp^m &\dashrightarrow & \dual \mbp^n \\
			& y & \mapsto & \mathcal{A}(\cdot, y, y) = (q_0(y), \ldots, q_n(y))
		\end{tabu}.
	\]

%%% SECTION
\section{Proof of the main result}

\subsection{Construction of $16$ points in $\mbp^4$}

Given a hyperelliptic genus $3$ curve $Y$ of the form
\[
	Y\: y^2 = f(z_0, z_1)
\]
such that $f(z_0, z_1)$ is homogeneous, square-free, and $f(0,1), f(1,0) \in \mbq^\timestwo$, we describe a method to construct $16$ points in $\mbp^4$ based on the construction in \cite[Section 6]{KulkarniVemulapalli}. 

Define $K := \mbq[z]/(f(z,1))$ and $L := \mbq[z]/(f(z^2,1))$. An elementary calculation shows that the conditions on $f(z_0, z_1)$ allow us to write
\[
	f(z_0, z_1) = b(z_0, z_1)^2 - z_0z_1^3c(z_0,z_1) = -
	\det 
	\begin{bmatrix}
		z_0 z_1^3 & b(z_0, z_1) \\
		b(z_0, z_1) & c(z_0, z_1)
	\end{bmatrix}
\]
for some homogeneous polynomials $b(z_0, z_1), c(z_0, z_1)$ of degree $4$. If $\alpha_1, \ldots, \alpha_8$ are the roots of $f(z, 1)$, then the set
\[
	\mathfrak{C} := \left\{(\pm\sqrt{\alpha_j} : (\pm\sqrt{\alpha_j})^{-1} \cdot b(\alpha_j, 1) : \alpha_j^2 : \alpha_j : 1) \ : \ 1 \leq j \leq 8\right\}
\]
of points in $\mbp^4$ is split over $\widehat L$, the Galois closure of $L/\mbq$. One can check that when the $\alpha_j$ are distinct, the Vandermonde matrix of degree $2$ forms evaluated at the points of $\mathfrak{C}$ has corank $5$, so $\mathfrak{C}$ is contained in the intersection of $4$ quadrics. Generically, $\mathfrak{C}$ is a complete intersection of $4$ quadrics in $\mbp^4$. If $L$ is presented as in Equation~\eqref{eq: L presentation} such that $\alpha_1, \ldots, \alpha_r$ are primitive elements of $K_1, \ldots, K_r$ (respectively), then we have that
%the collection of all roots $\pm \sqrt{\alpha_j}$ with $1 \leq j \leq 8$ defines a subset of the $\mbq$-homomorphisms $\sigma\: L \rightarrow \widehat L$ which we can identify with $\Omega$.
$\{\pm \sqrt{\alpha_j} : 1 \leq j \leq 8\} = \Hom_\mbq(L, \mbq^\sep)$. Thus, the permutation action of $\Gal(\mbq^\sep/\mbq)$ on the $16$ points of $\mathfrak{C}$ is identical to the action of $\Gal(\mbq^\sep/\mbq)$ on $\Omega$.

We can explicitly describe the linear space of quadrics containing $\mathfrak{C}$. Write $\mbp^4 := \Proj(\mbq[y_0, y_1, y_2, y_3, y_4])$, which admits a natural projection to $\mbp^2 := \Proj(\mbq[y_2, y_3, y_4])$. Under the Veronese embedding 
	\[
		\mapdef{\nu_2}{\mbp^1}{\Proj(\mbq[y_2, y_3, y_4])}{(z_0:z_1)}{(z_0^2 : z_0z_1 : z_1^2)}
	\]
the quartic forms $z_0z_1^3, b(z_0, z_1), c(z_0, z_1)$ can be identified with the three quadratic forms $y_3y_4$, $b(y_2, y_3, y_4)$, $c(y_2, y_3, y_4)$; the identified quadric forms are unique modulo $y_3^2 - y_2y_4$. The determinantal representation for $f(z_0, z_1)$ becomes a quadratic determinantal representation, i.e., is a tensor in $\Sym^2 \mbq^3 \otimes \Sym_2 \mbq^2$. 

\begin{construction} \label{cons: tensor Recillas}
	Let $\mbp^2 := \Proj(k[y_2, y_3, y_4])$. Given a conic $C \subset \mbp^2$ and a tensor $\mathcal{B} \in \Sym^2 k^3 \otimes \Sym_2 k^2$ representing a $2 \times 2$ matrix of quadratic forms, we can construct a tensor $\mathcal{A} \in k^4 \otimes (\Sym_2 k^2 \oplus \Sym_2 k^3)$ as follows: Writing $\mathcal{B}(\mathbf{y}, \cdot) = \begin{bmatrix} a & b \\ b & c \end{bmatrix}$, we have that
	\begin{align*}
		a(\mathbf{y}) = \sum_{2 \leq i,j \leq 4} a_{ij} y_iy_j, \quad
		b(\mathbf{y}) = \sum_{2 \leq i,j \leq 4} b_{ij} y_iy_j, \quad
		c(\mathbf{y}) = \sum_{2 \leq i,j \leq 4} c_{ij} y_iy_j.
	\end{align*}
	We may write the defining equation for $C$ as $\sum_{2 \leq i,j \leq 4} d_{ij} y_iy_j$. (The terms with $i \neq j$ appear twice.) Define the slices of $\mathcal{A}$ to be
	\begin{alignat*}{3}
		&A_0 := 
		\begin{bmatrix}
			-1 & 0 & 0 & 0 & 0\\
			0 & 0 & 0 & 0 & 0 \\
			0 & 0 & a_{22} & a_{23} & a_{24} \\
			 0 & 0 & a_{32} & a_{33} & a_{34} \\
			 0 & 0 & a_{42} & a_{43} & a_{44} \\
		\end{bmatrix},
		\quad 
		&&A_1 := 
		\begin{bmatrix}
			0 & -\frac{1}{2} & 0 & 0 & 0 \\
			-\frac{1}{2} & 0 & 0 & 0 & 0 \\
			0 & 0 & b_{22} & b_{23} & b_{24} \\
			 0 & 0 & b_{32} & b_{33} & b_{34} \\
			 0 & 0 & b_{42} & b_{43} & b_{44} \\
		\end{bmatrix},
		\\
		&A_2 := 
		\begin{bmatrix}
			0 & 0 & 0 & 0 & 0 \\
			0 & -1 & 0 & 0 & 0 \\
			0 & 0 & c_{22} & c_{23} & c_{24} \\
			 0 & 0 & c_{32} & c_{33} & c_{34} \\
			 0 & 0 & c_{42} & c_{43} & c_{44} \\
		\end{bmatrix},
		\quad
		&&A_3 :=
		\begin{bmatrix}
			0 & 0 & 0 & 0 & 0 \\
			0 & 0 & 0 & 0 & 0 \\
			0 & 0 & d_{22} & d_{23} & d_{24} \\
			 0 & 0 & d_{32} & d_{33} & d_{34} \\
			 0 & 0 & d_{42} & d_{43} & d_{44} \\
		\end{bmatrix}.			
	\end{alignat*}
	The tensor $\mathcal{A}$ generically defines a genus $4$ curve via an intersection of the two symmetric determinantal varieties
	$X_2 := Z(4x_0x_2 - x_1^2)$  and 
	$X_3 := Z\left(\det \mathcal{A}^{(2)}(\mathbf{x}, \cdot, \cdot) \right)$,
	where
	\[
		\mathcal{A}^{(2)}(\mathbf{x}, \cdot, \cdot) := 
			x_0 
			\begin{bmatrix}
				a_{22} & a_{23} & a_{24} \\
				 a_{32} & a_{33} & a_{34} \\
				 a_{42} & a_{43} & a_{44} \\
			\end{bmatrix}
			+ x_1
			\begin{bmatrix}
				b_{22} & b_{23} & b_{24} \\
				 b_{32} & b_{33} & b_{34} \\
				 b_{42} & b_{43} & b_{44} \\
			\end{bmatrix}
			+ x_2
			\begin{bmatrix}
				c_{22} & c_{23} & c_{24} \\
				 c_{32} & c_{33} & c_{34} \\
				 c_{42} & c_{43} & c_{44} \\
			\end{bmatrix}
			+ x_3
			\begin{bmatrix}
				d_{22} & d_{23} & d_{24} \\
				 d_{32} & d_{33} & d_{34} \\
				 d_{42} & d_{43} & d_{44} \\
			\end{bmatrix}.
	\]
	From the four symmetric matrices $A_0, \ldots, A_3$ we obtain an intersection of four quadrics in $\mbp^4$.
\end{construction}

\begin{comment}
\begin{example} \label{ex: main example 1}
	Let
	\[
		f(z_0, z_1) := 4z_0^8 + z_0^7z_1 + 67z_0^6z_1^2 + 63z_0^5z_1^3 + 58z_0^4z_1^4 + 100z_0^3z_1^5 + 32z_0^2z_1^6 + z_1^8 \in \mathbb{F}_{101}[z_0, z_1].
	\]
	Then the tensor given by the construction above, represented as a matrix of linear forms, is
	\[
	\left(\begin{array}{rrrrr}
	x_0 & 51x_1 & 0 & 0 & 0 \\
	51x_1 & x_2 & 0 & 0 & 0 \\
	0 & 0 & 100x_1 + 96x_2 & 0 & 51x_3 \\
	0 & 0 & 0 & 85x_1 + 100x_2 + 100x_3 & 50x_0 + 100x_1 + 50x_2 \\
	0 & 0 & 51x_3 & 50x_0 + 100x_1 + 50x_2 & 99x_1 + 94x_2
	\end{array}\right)
	\]
	and the resulting genus $4$ curve in $\mbp^3$ is defined by the common vanishing locus of
	\begin{align*}
	q(x) &:= x_0x_2 + 25x_1^2, \\
	g(x) &:= 76x_0^2x1 + 77x_0^2x_2 + x_0x_1^2 + 56x_0x_1x_2 + 53x_0x_2^2 + 70x_1^3 + 35x_1^2x_2 \\
	& \quad + 99x_1^2x_3 + 9x_1x_2^2 + 84x_1x_2x_3 + 4x_1x_3^2 + 42x_2^3 + 66x_2^2x_3 + 76x_2x_3^2 + 76x_3^3.
	\end{align*}
	Computer algebra can be used to verify that the genus $4$ curve is nonsingular. (See subsection~\ref{sec: sub: software}.)
\end{example}
\end{comment}

\begin{example} \label{ex: main example 1}
	Let
	\[
		f(z_0, z_1) := (z_0^2 - z_1^2)(z_0^2 - 4z_1^2)(z_0^2 - 9z_1^2)(z_0^2 - 16z_1^2) \in \mathbb{F}_{101}[z_0, z_1].
	\]
	Choosing
	\[
		b(z_0, z_1) := z_0^4 + 86z_0^2z_1^2 + 24z_1^4 \sim y_2^2 + 86y_3^2 + 24y_4^2, \qquad c(z_0, z_1) := -z_0z_1^3 \sim -y_3y_4,
	\]
	we have that $f(z_0, z_1) = b(z_0,z_1)^2 - z_0z_1^3 c(z_0,z_1)$. 
	The tensor given by Construction~\ref{cons: tensor Recillas}, represented as a matrix of linear forms, is
	\[
		\left(\begin{array}{rrrrr}
			-x_0 & -51x_1 & 0 & 0 & 0 \\
			-51x_1 & -x_2 & 0 & 0 & 0 \\
			0 & 0 & x_1 & 0 &  -51x_3 \\
			0 & 0 & 0 & 86x_1 + x_3 & 51(x_0-x_2) \\
			0 & 0 & -51x_3 & 51(x_0-x_2) & 24x_1
		\end{array}\right)
	\]
	and the resulting genus $4$ curve in $\mbp^3$ is defined by the common vanishing locus of
	\begin{align*}
	q(x) &:= x_0x_2 - x_1^2, \\
	g(x) &:= 25x_0^2x_1 + 51x_0x_1x_2 + 44x_1^3 + 24x_1^2x_3 + 25x_1x_2^2 + 29x_1x_3^2 + 25x_3^3.
	\end{align*}
	Computer algebra can be used to verify that the genus $4$ curve is nonsingular. (See subsection~\ref{sec: sub: software}.)
\end{example}

\begin{lemma} \label{lem: data for nonsingular genus 4 curve}
	Let $L/\mbq$ be an \'etale algebra arising from Construction~\ref{cons: etale algebra of degree 16}, let $K$ be the distinguished subalgebra, and let
	\[
		U := \left\{(\alpha_1, \ldots, \alpha_r) \in \prod_j K_j : L \iso \prod_j {K_j[x]}/{\gen{x^2-\alpha_j}}, \ K_j \iso \mbq(\alpha_j)\right\}.
	\]
	Let $V \subseteq U$ be the subset such that for any $\alpha \in V$, Construction~\ref{cons: tensor Recillas} applied to the characteristic polynomial of $\alpha$ produces a nonsingular genus $4$ curve and a base locus $\mathfrak{C}$ of dimension $0$ and degree $16$ of the linear space of quadrics. Then $V$ is non-empty.
\end{lemma}

\begin{proof}
	Example~\ref{ex: main example 1} shows that there is a non-empty open subscheme of $\Poleight$ where the constructed curve is nonsingular. The result follows from Corollary~\ref{cor: sqrt primitive basis}.
\end{proof}

%%% SUBSECTION %%%
\subsection{Cycles on double covers}

In \cite{Reid1972quadrics}, Reid showed that the algebraic cycles on the complete intersection of 3 quadrics correspond to algebraic cycles within the Prym variety of the natural double cover of the degeneracy locus. An alternative reference for this construction is \cite{Tyurin1975}. In this section, we show how a similar construction allows us to identify pairs of points on the intersection of $4$ quadrics in $\mbp^4$ with the $112$ exceptional curves on a del Pezzo surface of degree one.

The following lemma contains some results we will freely use about quadrics.

\begin{lemma} \label{lem: properties of quadrics}
	Let $k$ be a field of characteristic not equal to $2$, let $A \in \Sym_2 k^{m+1}$ be of rank $r>0$, and let $Q := Z(\mathbf{y}^TA\mathbf{y}) \subset \mbp_k^m$ be the associated quadric. Then:
	\begin{enumerate}[(a)]
		\item
		Any singular quadric is a cone over a non-singular quadric.
		\item
		The singular locus of $Q$ is contained in every maximal isotropic subspace. In other words, every maximal isotropic subspace is a cone over an isotropic subspace of the nonsingular part. Furthermore, the singular locus of $Q$ is the linear subspace $\mbp(\ker A) \subset \mbp^m$.	
		\item
		Over $\bar{k}$, the dimension of a maximal isotropic subspace is $\floor{\frac{r}{2}} + \dim(\ker A) - 1$.
		\item
		If $r$ is even, there are two distinct families of maximal isotropic subspaces of $Q$. If $r$ is odd, then there is a unique family of maximal isotropic subspaces.
		\item
		If $r$ is even, and $V,W$ are maximal, then $V,W$ are contained in the same family if and only if $\dim(V) - \dim (V \cap W) \equiv 0 \pmod 2$.
	\end{enumerate}
\end{lemma}

\begin{proof}
	See \cite{GriffithsHarris1994}.
\end{proof}

If $\mathcal{L}$ is a maximal isotropic subspace of a quadric $Q$, we will denote by $[\mathcal{L}]$ the family of maximal isotropic subspaces containing $\mathcal{L}$. 

\begin{proposition} \label{prop: defining the generator cover}
	Let $S_r^m$ be the space of quadrics of even rank $r$. Let $x \in S_r^m$ and $[\mcL]$ denote a family of maximal isotropic subspaces of $x$. Then the choice of generator $(x,[\mcL]) \mapsto x$ defines a nontrivial double cover branched over $S_{r-1}^m$. 
\end{proposition}

\begin{proof}
	See \cite[Section 5]{Tyurin1975}.
\end{proof}

\begin{definition}
	The \emph{generator double cover} is the morphism $\gencover\: (x,[\mcL]) \mapsto x$ given by Proposition~\ref{prop: defining the generator cover}.
\end{definition}

The generator double cover allows us to identify the secants between $16$ points in $\mbp^4$ with exceptional curves of a del Pezzo surface of degree $1$. In our particular case of a block-diagonal tensor $\mathcal{A} \in k^4 \otimes (\Sym_2 k^2 \oplus \Sym_2 k^3)$, the locus of quadrics of rank at most $4$ is $X_5 := Z(\det \mathcal{A}(\mathbf{x}, \cdot, \cdot))$, which is the union of the quadric cone $X_2$ where the first block degenerates and the symmetroid cubic $X_3$ where the second block degenerates. The generator double cover $\gencover\: \wtilde X_5 \rightarrow X_5$ is branched over the locus of quadrics of rank at most $3$, which consists of the genus four curve $X = X_2 \cap X_3$ as well as the singularities of $X_2$ and $X_3$. Generically, the web of quadrics defined by $\mathcal{A}$ does not contain any quadrics of rank $2$. 
We obtain by restriction a double cover $\gencover\: \wtilde X_2 \rightarrow X_2$ branched along a genus $4$ curve and the vertex of the cone; in other words, $\wtilde X_2$ is a del Pezzo surface of degree one.

Let $\mathfrak{C}$ be the intersection of all the quadrics in the web defined by $\mathcal{A}$, which we saw before is generically a complete intersection of four quadrics in $\mbp^4$. If $\ell$ is a secant of $\mathfrak{C}$ and $V_\ell \subset k^5$ is the affine cone over $\ell$, then we define a subscheme of $\wtilde X_2$ by
	\[
		\tau(\ell) := \left\{(x, [\mathcal{L}]) \in \wtilde X_2 : \mathcal{L} := V_\ell + \ker \mathcal{A}(x, \cdot, \cdot) \text{ is a maximal isotropic subspace of } \mathcal{A}(x, \mathbf{y}, \mathbf{y}) \right\}.
	\]
Observe that $\tau(\ell)$ is well-defined; the kernel of a quadric $Q$ of rank $4$ generically does not meet $V_\ell$, so the space $V_\ell + \ker Q$ is a maximal isotropic subspace of $Q$. Notice that the automorphism $\eta\: \mbp^4 \rightarrow \mbp^4$ defined by $\eta\: (y_0:y_1:y_2:y_3:y_4) \mapsto (-y_0:-y_1:y_2:y_3:y_4)$ acts invariantly on all of the quadratic forms in the web defined by $\mathcal{A}$. 

If $\mathcal{A}(x, \mathbf{y}, \mathbf{y})$ is a quadratic form of rank $4$ vanishing on $\ell$, then it must also vanish on $\eta(\ell)$. 
%It turns out that $\ell$ meets $\mbp(\ker \mathcal{A}(x, \cdot, \cdot))$ at a single point (see \cite[???]{KulkarniVemulapalli}).
%
%Furthermore, since $\ell$ and $\eta(\ell)$ do not (generically) intersect, the maximal isotropic subspaces $V_\ell + \ker \mathcal{A}(x, \cdot, \cdot)$ and $\eta(V_\ell) + \ker \mathcal{A}(x, \cdot, \cdot)$ intersect only along the kernel, and thus lie in opposite families of maximal isotropic subspaces. 
Furthermore, the intersection of the maximal isotropic subspaces $V_\ell + \ker \mathcal{A}(x, \cdot, \cdot)$ and $\eta(V_\ell) + \ker \mathcal{A}(x, \cdot, \cdot)$ is generically $\ker \mathcal{A}(x, \cdot, \cdot)$, and thus the two maximal isotropic subspaces lie in opposite families.
In other words, the curves $\tau(\ell), \tau(\eta \ell)$ on $\wtilde X_2$ only intersect along the branch locus and are exchanged by $\Aut(\wtilde X_2/X_2)$.

The last result we need is a theorem from \cite{KulkarniVemulapalli}. To clarify the statement of the theorem, any nonhyperelliptic genus $4$ curve $X$ with a vanishing even theta characteristic $\theta_0$ has a unique vanishing even theta characteristic. This allows us to partition the $2$-torsion classes of the Jacobian variety into one of two types:
\begin{itemize}
	\item (odd) $\epsilon \in \Jac(X)[2]$ is of the form $[\theta- \theta_0]$ for some odd theta characteristic of $X$.
	\item (even) $\epsilon \in \Jac(X)[2]$ is not even.
\end{itemize}
There are $120$ odd $2$-torsion classes and $135$ nontrivial even $2$-torsion classes.

\begin{theorem}[{\cite[Theorem~1.1.3]{KulkarniVemulapalli}}]
	\label{thm: genusfoursummary}
	Let $k$ be a field of characteristic not $2$ or $3$. Then:
	\begin{enumerate}[(a)]
	\item
	There is a canonical bijection between:
	\[
		\left \{  \begin{array}{c}
		k \text{-isomorphism classes of tuples } (X, \epsilon, \theta_0),\\
		\text{where } X \text{is a smooth genus $4$ curve with vanishing even theta} \\
		\text{characteristic } \theta_0 \text{ with a rational divisor class defined over } k, \\
		\text{and } \epsilon \text{ is a nontrivial even $2$-torsion class}
		\end{array}
		 \right \} 	 \longleftrightarrow \left \{ \begin{array}{c}
		 \text{nondegenerate orbit classes of } \\
		 k^4 \otimes \left(\Sym_2 k^2 \oplus \Sym_2 k^3 \right) \\
		 \text{under the action of } \\
		 \GL_4(k) \times \GL_2(k) \times \GL_3(k) 
		 \end{array} \right \}
	\]
	\end{enumerate}
	\noindent
	Let $\mathcal{A} \in k^4 \otimes (\Sym_2 k^2 \oplus \Sym_2 k^3)$ be a nondegenerate tensor and let $\theta_0$ and $\epsilon$ be the associated line bundles on $X$. Then:
	\begin{enumerate}[(b)]
	\item
	The images of the $120$ secants of $\mathfrak{C}$ under $\psi$ define $56 + 8$ tritangent planes of $X$. Viewing $X \cap H$ as a divisor of $X$, eight of these tritangents satisfy $X \cap H = 2D$ where $D \in |\theta_0|$. The other $56$ tritangents satisfy $X \cap H = 2D$, where $D$ is the effective representative of an odd theta characteristic of $X$.
	
	\item
	Let $e_2$ denote the Weil pairing on $\Jac(X)[2]$. Then the $56$ distinct odd theta characteristics constructed from the secants of $\mathfrak{C}$ are precisely the odd theta characteristics $\theta$ such that $e_2(\theta \otimes \theta_0^\vee, \epsilon) = 0$.
	\end{enumerate}

\end{theorem}

\begin{comment}
	\label{thm: genusfoursummary}

	Let $\mathcal{A} \in k^4 \otimes (\Sym_2 k^2 \oplus \Sym_2 k^3)$ be a nondegenerate tensor and let $\theta_0$ and $\epsilon$ be the associated line bundles on $X$. Then:
	\begin{enumerate}[(a)]
	\item
	The images of the $120$ secants of $\mathfrak{C}$ under $\psi$ define $56 + 8$ tritangent planes of $X$. Viewing $X \cap H$ as a divisor of $X$, eight of these tritangents satisfy $X \cap H = 2D$ where $D \in |\theta_0|$. The other $56$ tritangents satisfy $X \cap H = 2D$, where $D'$ is the effective representative of an odd theta characteristic of $X$.
	
	\item
	Let $e_2$ denote the Weil pairing on $\Jac(X)[2]$. Then the $56$ distinct odd theta characteristics constructed from the secants of $\mathfrak{C}$ are precisely the odd theta characteristics $\theta$ such that $e_2(\theta \otimes \theta_0^\vee, \epsilon) = 0$.
	\end{enumerate}
\end{comment}

The inclusion $X \inj \wtilde X_2$ induces a restriction morphism of divisor classes. We denote by $\Pic(\wtilde X_2)^\perp$ the divisor classes orthogonal to the canonical class under the intersection pairing. There is an exact sequence
\[
	\xym{
		0 \ar[r] & 2\Pic(\wtilde X_2)^\perp \ar[r] & \Pic(\wtilde X_2)^\perp \ar[r] & \Pic(X)[2] \ar[r] & 0.
	}
\]
Additionally, the anti-canonical class of $\wtilde X_2$ restricts to the unique vanishing even theta characteristic $\theta_0$ of $X$. If $e \in \Pic(\wtilde X_2)$ is an exceptional curve and $\kappa$ is the canonical class, then $e + \kappa \in \Pic(\wtilde X_2)^\perp$, the restriction of $e$ to $X$ defines an odd theta characteristic, and the sublattice generated by
\[
	R_{D_8} := \{ e + \kappa \in \Pic(\wtilde X_2) : e^2 = -1, \ \ e_2((e + \kappa)|_{X}, \epsilon) = 0\}
\]
is a root lattice of type $D_8$; the set of roots of this lattice is precisely $R_{D_8}$.

\begin{proposition} \label{prop: Identification of $G$-sets}
	Let $\ell$ be a secant of $\mathfrak{C}$. Then the curve $\tau(\ell)$ is one of the $112$ exceptional curves of $\wtilde X_2$ with the property that
	\[
		e_2((\tau(\ell) + \kappa)|_{X}, \epsilon) = 0.
	\]
	Furthermore, each of the $112$ exceptional curves of this type is of the form $\tau(\ell)$ for some secant of $\mathfrak{C}$.
\end{proposition}

\begin{proof}
	Follows immediately from Theorem~\ref{thm: genusfoursummary} and the discussion above.
\end{proof}

The automorphism $\eta$ restricts to an automorphism of $\mathfrak{C}$. A pair of geometric points of $\mathfrak{C}$ is either:
\begin{itemize}
	\item Type $8$: a pair of the form $\{p, \eta(p)\}$, or
	
	\item Type $112$: a pair not of the form $\{p, \eta(p)\}$.
\end{itemize}
We see that any $G$-action preserves the type of a pair $\{p,q\}$. In particular, 
Proposition~\ref{prop: Identification of $G$-sets} shows that the $G$-set of pairs of type $112$ is isomorphic (as a $G$-set) to the $112$ roots of the $D_8$ lattice identified in Proposition~\ref{prop: Identification of $G$-sets}. 

%%% SUBSECTION
\subsection{Proof of the main theorem}

Let $\rho\: \Gal(\mbq^\sep/\mbq) \rightarrow W_{D_8}$ be a continuous homomorphism with image $G$. By Construction~\ref{cons: etale algebra of degree 16}, we construct an \'etale algebra $L$ of degree $16$ on which Galois acts via $G$, as well as its distinguished subalgebra $K$ of degree $8$.

We may choose an element $(\alpha_1, \ldots, \alpha_r) \in V \subset K$ as in Lemma~\ref{lem: data for nonsingular genus 4 curve} and let $f$ be its characteristic polynomial over $\mbq$. Construction~\ref{cons: tensor Recillas} allows us to construct a locus of $16$ points with a $G$-action from $f$, as well as a nonsingular genus $4$ curve $X$ contained in a quadric cone $X_2$. The domain of the generator double cover $\gencover\: \wtilde X_2 \rightarrow X_2$ is a del Pezzo surface of degree one, and the Galois action on the $112$ roots of the $D_8$ sublattice from Proposition~\ref{prop: Identification of $G$-sets} has the Galois action prescribed by $\rho$.

%Theorem~\ref{thm: main theorem} is implied by the solution to the inverse Galois problem for $16$ points in $\mbp^4$ defined by a tensor $\mathcal{A} \in k^4 \otimes (\Sym_2 k^2 \oplus \Sym_2 k^3)$ whose associated genus four curve is smooth. However, we have solved this inverse Galois problem already, so we are done.

%%%%%%%%%%%%%%%%%%%%%%
% BIBLIOGRAPHY
%%%%%%%%%%%%%%%%%%%%%%

%% Kill the MR tag in the bibliography.	
\renewcommand{\MR}[1]{}


\begin{bibdiv}
	\begin{biblist}
		
		\bib{Magma}{article}{
		   author={Bosma, Wieb},
		   author={Cannon, John},
		   author={Playoust, Catherine},
		   title={The Magma algebra system. I. The user language},
		   note={Computational algebra and number theory (London, 1993)},
		   journal={J. Symbolic Comput.},
		   volume={24},
		   date={1997},
		   number={3-4},
		   pages={235--265},
		   issn={0747-7171},
		   review={\MR{1484478}},
		   doi={10.1006/jsco.1996.0125},
		}

		\bib{Corn2007}{article}{
		   author={Corn, Patrick},
		   title={The Brauer-Manin obstruction on del Pezzo surfaces of degree 2},
		   journal={Proc. Lond. Math. Soc. (3)},
		   volume={95},
		   date={2007},
		   number={3},
		   pages={735--777},
		   issn={0024-6115},
		   review={\MR{2368282}},
		   doi={10.1112/plms/pdm015},
		}
					
		\bib{Dol2012}{book}{
			author={Dolgachev, Igor V.},
			title={Classical algebraic geometry},
			note={A modern view},
			publisher={Cambridge University Press, Cambridge},
			date={2012},
			pages={xii+639},
			isbn={978-1-107-01765-8},
			review={\MR{2964027}},
			doi={10.1017/CBO9781139084437},
		}	

		\bib{ElsenhansJahnel2019}{article}{
		   author={Elsenhans, Andreas-Stephan},
		   author={Jahnel, J\"{o}rg},
		   title={On plane quartics with a Galois invariant Cayley octad},
		   journal={Eur. J. Math.},
		   volume={5},
		   date={2019},
		   number={4},
		   pages={1156--1172},
		   issn={2199-675X},
		   review={\MR{4015450}},
		   doi={10.1007/s40879-018-0292-3},
		}

		\bib{GriffithsHarris1994}{book}{
		   author={Griffiths, Phillip},
		   author={Harris, Joseph},
		   title={Principles of algebraic geometry},
		   series={Wiley Classics Library},
		   note={Reprint of the 1978 original},
		   publisher={John Wiley \& Sons, Inc., New York},
		   date={1994},
		   pages={xiv+813},
		   isbn={0-471-05059-8},
		   review={\MR{1288523}},
		   doi={10.1002/9781118032527},
		}
		
		\bib{KlunersMalleDatabaseArticle}{article}{
		   author={Kl\"{u}ners, J\"{u}rgen},
		   author={Malle, Gunter},
		   title={A database for field extensions of the rationals},
		   journal={LMS J. Comput. Math.},
		   volume={4},
		   date={2001},
		   pages={182--196},
		   review={\MR{1901356}},
		   doi={10.1112/S1461157000000851},
		}
		
		\bib{KlunersMalleDatabase2021}{article}{
		   author={Kl\"{u}ners, J\"{u}rgen},
		   author={Malle, Gunter},		
		   title={A Database for Number Fields},
		   date={2021},
		   note={online database},
		   eprint={http://galoisdb.math.upb.de/home}
		}

		\bib{KulkarniVemulapalli}{article}{
			author={Kulkarni, Avinash},
			author={Vemulapalli, Sameera},
			title={On Intersections of symmetric determinantal varieties and theta characteristics of canonical curves},
			date={2021}
			eprint={arXiv:2109.08740}
		}

		\bib{milneFT}{book}{
			author={Milne, James S.},
			title={Fields and Galois Theory (v5.00)},
			year={2021},
			note={Available at www.jmilne.org/math/},
			pages={142}
		}

		\bib{Reid1972quadrics}{thesis}{
			author = {Reid, Miles}, 
			title = {The complete intersection of two or more quadrics}, 
			note = {Ph.D. dissertation},
			publisher = {Cambridge University},
			year = {1972}
		}
						
		\bib{Tyurin1975}{article}{
			author = {Tyurin, A. N.},
			title = {On intersections of quadrics},
			year = {1975},
			journal = {Russian Mathematical Surveys},
			doi = {http://dx.doi.org/10.1070/RM1975v030n06ABEH001530},
			volume = {30},
			issue = {6}
		}

	\end{biblist}
\end{bibdiv}
\end{document}